\documentclass[twoside,a4paper,11pt]{amsart}

\usepackage{amsmath,amssymb,amsthm,amscd}

\usepackage[]{latexsym,amssymb,amsmath,amsfonts, amsthm, verbatim}

\usepackage[all,cmtip,ps]{xy}

\usepackage[latin1]{inputenc}

\usepackage[english]{babel}

\def\projdim{\mathop{{\rm pd}}}

\usepackage{graphicx}

\newcommand{\add}{\mathrm{add}}

\newcommand{\rmod}{\mathrm{Mod-}}

\newcommand{\rfmod}{\mathrm{mod-}}

\newcommand{\coh}{\mathrm{coh}}

\newcommand{\ma}{\mbox{\rm mod-$A$}}

\newcommand{\mb}{\mbox{\rm mod-$B$}}

\newcommand{\mc}{\mbox{\rm mod-$C$}}

\newcommand{\mh}{\mbox{\rm mod-$H$}}

\newcommand{\dba}{D^b(\ma)}

\newcommand{\dbb}{D^b(\mb)}

\newcommand{\dbc}{D^b(\mc)}

\newcommand{\MA}{\mbox{\rm Mod-$A$}}

\newcommand{\MB}{\mbox{\rm Mod-$B$}}

\newcommand{\MC}{\mbox{\rm Mod-$C$}}

\newcommand{\dua}{D(\MA)}

\newcommand{\dub}{D(\MB)}

\newcommand{\duc}{D(\MC)}

\newcommand{\Z}{\mathbb{Z}}

\DeclareMathOperator{\Hom}{Hom}

\DeclareMathOperator{\End}{End}

\DeclareMathOperator{\Ext}{Ext}

\DeclareMathOperator{\Img}{Im}

\DeclareMathOperator{\Coker}{Coker}

\DeclareMathOperator{\gldim}{gl. dim}

\DeclareMathOperator{\sgldim}{s.gl.dim}

\DeclareMathOperator{\fdim}{fdim}

\DeclareMathOperator{\length}{length}

\newcommand{\Dcal}{\ensuremath{\mathcal{D}}}

\newcommand{\Ccal}{\ensuremath{\mathcal{C}}}

\newcommand{\Bcal}{\ensuremath{\mathcal{B}}}

\newcommand{\Ycal}{\ensuremath{\mathcal{Y}}}

\newcommand{\Xcal}{\ensuremath{\mathcal{X}}}

\newcommand{\Hcal}{\ensuremath{\mathcal{H}}}

\newcommand{\Tria}{\mathrm{Tria\,}}

\newcommand{\tria}{\mathrm{tria\,}}

\theoremstyle{plain}

\newtheorem{thm}{Theorem}[section]

\newtheorem{prop}[thm]{Proposition}

\newtheorem{lem}[thm]{Lemma}

\newtheorem{cor}[thm]{Corollary}

\newtheorem{ex}[thm]{Example}

\theoremstyle{definition}

\theoremstyle{remark}

\newtheorem*{rem}{Remark}

\newcommand{\ra}{\rightarrow}

\newcommand{\les}{\leqslant}

\newcommand{\ten}{\otimes}

\newcommand{\lten}{\overset{\boldmath{L}}{\ten}}

\newcommand{\rhom}{\boldmath{R}\mathrm{Hom}}

\newcommand{\mcy}{\mathcal{Y}}

\newcommand{\mcx}{\mathcal{X}}

\begin{document}

\begin{center}

{\bf Jordan H\"older theorems for derived module categories of
piecewise hereditary algebras}


\bigskip

{\sc Lidia Angeleri  H\" ugel, Steffen Koenig, Qunhua
Liu}

\bigskip


\end{center}

\address{Lidia Angeleri  H\" ugel
\\ Dipartimento di Informatica - Settore Matematica
\\ Universit\`a degli Studi di Verona
\\ Strada Le Grazie 15 - Ca' Vignal 2
\\ I - 37134 Verona, Italy}
\email{lidia.angeleri@univr.it}

\address{Steffen Koenig, Qunhua Liu\\ Mathematisches Institut der
Universit\"at Stuttgart \\ Pfaffenwaldring 57 \\ 70569 Stuttgart, Germany}
\email{skoenig@mathematik.uni-stuttgart.de, qliu@mathematik.uni-stuttgart.de}



\date{\today}


\bigskip

\begin{quote}

{\footnotesize {\bf Abstract.} A Jordan H\"older theorem is established
for derived module categories of piecewise hereditary algebras. The
resulting composition series of derived categories are shown to be
independent of the choice of bounded or unbounded derived module categories, and
also of the choice of finitely generated or arbitrary modules.}

\end{quote}

\vspace{4ex}

\section*{Introduction}

Jordan H\"older theorems are classical and fundamental results in
group theory and in module theory. Under suitable assumptions, a
Jordan H\"older theorem asserts the existence of a finite `composition series',
the subquotients of which are `simple' objects. A Jordan H\"older
theorem can be formulated when the concept of `short exact sequence'
has been defined. Then an object may be called simple if it is not
the middle term of a short exact sequence, that is, it is not an
extension of another two objects in the given class of objects
(groups, modules, \dots). Then finite series of short exact sequences
can be considered, where the given object is the middle term of the
first sequence, the end terms of the first sequence are middle terms
of further sequences, and so on, until simple objects are reached and
the process stops. A Jordan H\"older theorem states finiteness of this process
and the uniqueness of the simple constituents,
up to a suitable notion of isomorphism.

\medskip

About twenty years ago, the work of Cine, Parshall and Scott \cite{CPS} on
highest weight categories and quasi-hereditary algebras and on
`stratifications' of their derived module categories - i.e. on composition
series in our terminology - provided a first motivation to ask for
a Jordan H\"older theorem for derived categories
of rings, in the following sense: A `short exact sequence' of derived
categories is, by definition, a recollement of triangulated categories,
as defined by Beilinson, Bernstein and Deligne \cite{BBD}, with all three
triangulated categories being derived categories of rings. A derived
category is called `simple' if it does not admit a non-trivial
recollement. Wiedemann \cite{W} and Happel \cite{HapFib} found non-trivial examples of
`simple' derived categories. Only recently, however, a first Jordan H\"older
theorem could be established: Using methods developed in \cite{AKL,NZ},
a Jordan H\"older theorem has been provided in \cite{AKL2} for unbounded
derived module categories of hereditary artinian algebras.
Subsequently, Liu and Yang \cite{LY} have shown that
blocks of group algebras of finite groups always are derived simple.

\medskip

In \cite{AKL2} the problem has been made more precise by showing that no positive
answer can be expected when admitting arbitrary triangulated categories
as factors in composition series of derived module categories. Moreover, it
has been pointed out that an answer may depend on the choice of derived
categories one is working with - unbounded, left bounded or bounded - and
of the underlying module category - finitely generated or arbitrary modules.
Examples given in \cite{AKLY} show that these choices really matter and in
particular do have an effect on derived simpleness.

\medskip

Jordan H\"older theorems may fail for two reasons: Composition series may not
be finite - an example has been given in \cite{AKL2} - and composition series
may not be unique. The second point is much more subtle; rather sophisticated
examples recently have been constructed by Chen and Xi \cite{CX}; the algebras
there are not artinian.

\bigskip

The aim of the present article is to extend and to complement
the results of \cite{AKL2} in at least two ways:

\begin{itemize}

\item We extend the range of validity of the Jordan H\"older theorem from hereditary algebras to
piecewise hereditary algebras. These also include  quasi-tilted algebras
that are not related to hereditary algebras, but to hereditary abelian categories of
a geometric nature - coherent sheaves over weighted projective lines - and the corresponding
`canonical' algebras.

\item We show that the `same' composition series are obtained when considering unbounded
or bounded derived categories, finitely generated or arbitrary modules.

\end{itemize}

\medskip

{\bf Main Theorem.} {\em The (bounded or unbounded) derived category (using finitely generated
or arbitrary modules) of a finite dimensional piecewise hereditary algebra has a
finite composition series. The simple compositon factors are derived categories of
vector spaces over skew-fields: the endomorphism rings of the simple modules. These
composition factors are unique up to ordering and Morita equivalence.}

\bigskip

In \cite{AKL2}, the Jordan H\"older theorem for hereditary algebras actually has been
proven in a stronger form: Any composition series can be brought into a `normal form',
which means that the composition series is associated with a series of homological
epimorphisms, starting from the given algebra. This strong version is valid in the
present more general context, too. In order to establish it, we are using the results of
\cite{AKL2}. The proof of the Main Theorem stated above does, however, not use the
special case of it shown in \cite{AKL2}, for which we give an alternative proof here.
A new ingredient compared to \cite{AKL2} is the concept of strong global dimension,
recently investigated by Ringel \cite{RComp} and by Happel and Zacharia \cite{HZ}.
Other key ingredients are constructions of recollements for $D^b(\mathrm{mod})$ and, for
hereditary algebras, bijections relating recollements on different levels with each
other and with further data such as exceptional objects (Theorem \ref{propbij}) and
homological epimorphisms (Theorem \ref{tiltreco}, Theorem \ref{propbij}).

\bigskip

The organisation of this article is as follows: A preliminary first section recalls
definitions and concepts to be used later on. The second section discusses the
existence of recollements in general and the third section constructs recollements
from tilting modules. In section four a collection of positive and negative
examples is presented. In the final fifth section we prove the Main Theorem, which
is split into several results, giving more detail than the version stated above.

\bigskip

\section{Preliminaries}

Throughout this article, algebras are finite dimensional over a field. The
reason for this restriction - when comparing to  \cite{AKL}, where more generally
artinian algebras have been investigated - is that the theory of weighted projective
lines and corresponding canonical algebras is available over fields only.

\medskip

Let $A$ be a finite dimensional algebra over a
field $k$. Then $\ma$ denotes the category of finite dimensional right
$A$-modules, and $\MA$ the category of all right $A$-modules. Let
$\dba$ be the bounded derived category of $\ma$, and $\dua$ be
unbounded derived category of $\MA$.

\medskip

\subsection{Recollements}

Let $\Xcal, \Ycal$ and $\Dcal$ be triangulated categories. $\Dcal$
is said to be  a {\em recollement} (\cite{BBD}, see also
\cite{PS}) of $\Xcal$ and $\Ycal$ if there are six triangle
functors as in the following diagram

\[\xy (-50,0)*{\mathcal Y}; {\ar (-27,3)*{}; (-43,3)*{}_{i^{\ast}}};
{\ar (-43,0)*{}; (-27,0)*{}|{ i_{\ast}=i_!}}; {\ar(-27,-3)*{};
(-43,-3)*{}^{i^!}}; (-20,0)*{\mathcal D}; {\ar (3,3)*{};
(-11,3)*{}_{j_!}}; {\ar (-11,0)*{}; (3,0)*{}|{j^!=j^\ast}};
{\ar(3,-3)*{}; (-11,-3)*{}^{j_{\ast}}}; (10,0)*{\mathcal X};
\endxy\]


such that
\begin{enumerate}
\item $(i^\ast,i_\ast)$,\,$(i_!,i^!)$,\,$(j_!,j^!)$ ,\,$(j^\ast,j_\ast)$
are adjoint pairs;


\item  $i_\ast,\,j_\ast,\,j_!$  are full embeddings;


\item  $i^!\circ j_\ast=0$ (and thus also $j^!\circ i_!=0$ and
$i^\ast\circ j_!=0$);


\item  for each $C\in \Dcal$ there are triangles $$i_! i^!(C)\to
C\to j_\ast j^\ast (C)\to $$
 $$j_! j^! (C)\to C\to i_\ast i^\ast(C)\to$$

\end{enumerate}

Two recollements

\[\xy (-46,0)*{\mathcal Y}; {\ar (-29,3)*{}; (-43,3)*{}_{i^{\ast}}};
{\ar (-43,0)*{}; (-29,0)*{}|{ i_{\ast}=i_!}}; {\ar(-29,-3)*{};
(-43,-3)*{}^{i^!}}; (-25,0)*{\mathcal D}; {\ar (-7,3)*{};
(-21,3)*{}_{j_!}}; {\ar (-21,0)*{}; (-7,0)*{}|{j^!=j^\ast}};
{\ar(-7,-3)*{}; (-21,-3)*{}^{j_{\ast}}}; (-3,0)*{\mathcal X};
(10,0)*{\text{and}}; (24,0)*{\mcy'}; {\ar (41,3)*{};
(27,3)*{}_{i'^{\ast}}}; {\ar (27,0)*{}; (41,0)*{}|{
i'_{\ast}=i'_!}}; {\ar(41,-3)*{}; (27,-3)*{}^{i'^!}};
(45,0)*{\Dcal}; {\ar (63,3)*{}; (49,3)*{}_{j'_!}}; {\ar (63,0)*{};
(49,0)*{}|{j'^!=j'^\ast}}; {\ar(63,-3)*{};
(49,-3)*{}^{j'_{\ast}}}; (67,0)*{\mathcal X'};
\endxy
\]












are said to be {\em equivalent}, if the essential images of $i_\ast$ and
$i'_\ast$, of $j_\ast$ and $j'_\ast$, and of $j_!$ and $j'_!$
coincide, respectively.



\medskip

\subsection{Perpendicular categories, compact objects, tilting objects}

\label{tilting}

Given a triangulated category $\Ccal$ and an object $M$ in $\Ccal$, the
smallest triangulated full subcategory of $\Ccal$ containing $M$
and  closed under taking direct summands is denoted by $\tria(M)$.
When $\Ccal$ has small coproducts, the corresponding subcategory
closed under taking small coproducts is denoted by $\Tria(M)$. The
{\em perpendicular category} of $M$  in $\Ccal$, denoted by
$M^\bot$, is by definition the full subcategory of $\Ccal$
containing of those objects $X$ perpendicular to $M$, that is,
$\Hom_{\Ccal}(M,X[n])=0$ for all integers $n$.

\medskip

Let $A$ be a finite dimensional algebra. Write $P_A$ for the
category of finitely generated projective $A$-modules. It is well
known that the bounded homotopy category $K^b(P_A)$ coincides with
the subcategory $\tria(A)$ of $\dba$. It is called the {\em compact}
or {\em perfect subcategory} of $\dba$ and $\dua$. Its objects are
called {\em compact} or {\em perfect}. We identify a compact object
$X$ with its {\em `minimal $K^b(P_A)$-representative'},
i.e. a complex in $K^b(P_A)$, isomorphic to $X$, without direct
summands of the form $P\xrightarrow{\text{Id}} P$ or its shifts, for
some $P\in P_A$.

\medskip

Recall that an $A$-module $T\in \ma$ is a {\em tilting module}, if
the following hold: \begin{enumerate}
\item $\projdim(T)\les 1$;  \item $\Ext^1(T,T)=0$;
\item  There exists a short exact sequence $0\ra A\ra T_0\ra
T_1\ra 0$ where $T_i\in \add(T)$.
\end{enumerate}

\medskip

A complex $X$ in $\dba$ or $\dua$ is said to be {\em exceptional} if
it has no nontrivial self-extension, i.e. $\Hom(X,X[n])=0$ for all
integers $n$. It is said to be a {\em partial tilting complex}, if
it is exceptional and compact, and a {\em tilting complex}, if in
addition it generates the perfect subcategory, i.e. $\tria(X) = K^b(P_A)$.

\medskip

\subsection{Homological epimorphisms}\label{homoepi}

Recall that a ring homomorphism $\varphi:A\ra B$ is a {\em ring
epimorphism} if and only if the induced functor $\varphi_*: \mb
\ra \ma$ is a full embedding. Furthermore $\varphi$ is a {\em
homological epimorphism} if and only if the induced functor
$\varphi_*:\dbb\ra \dba$ is a full embedding, or equivalently
$\varphi_*: \dub \ra \dub$ is a full embedding (cf. \cite[Theorem
4.4]{GL}). In this case $\varphi$ induces a $D^b(\rfmod)$ level
recollement

\[\xy (-61,0)*{\dbb}; {\ar (-34,3)*{}; (-50,3)*{}_{i^{\ast}}}; {\ar
(-50,0)*{}; (-34,0)*{}|{ i_{\ast}=i_!}}; {\ar(-34,-3)*{};
(-50,-3)*{}^{i^!}}; (-22,0)*{\dba}; {\ar (3,3)*{}; (-11,3)*{}};
{\ar (-11,0)*{}; (3,0)*{}}; {\ar(3,-3)*{}; (-11,-3)*{}};
(8,0)*{\mcx};
\endxy
\]

and a $D(\rmod)$ level recollement

\[
\xy (-61,0)*{\dub}; {\ar (-34,3)*{}; (-50,3)*{}_{i^{\ast}}}; {\ar
(-50,0)*{}; (-34,0)*{}|{ i_{\ast}=i_!}}; {\ar(-34,-3)*{};
(-50,-3)*{}^{i^!}}; (-22,0)*{\dua}; {\ar (3,3)*{}; (-11,3)*{}};
{\ar (-11,0)*{}; (3,0)*{}}; {\ar(3,-3)*{}; (-11,-3)*{}};
(8,0)*{\mcx'};
\endxy
\]

 for some triangulated
categories $\mcx$ and $\mcx'$, and the functors on the left hand
side are induced by $\varphi$, that is, $i^* = -\lten_A B$, $i^! =
\rhom_A(B,-)$ and $i_* = \varphi_*$.

\medskip

We will be interested in the case when $\Xcal$ or $\Xcal'$ is a
derived module category.





\medskip

\subsection{Invariants of recollements} \label{Grothendieck}

Suppose there is a recollement
$$\xymatrix@!=5pc{\dbb \ar[r] &\dba \ar@<+1.5ex>[l] \ar@<-1.5ex>[l] \ar[r] &\dbc \ar@<+1.5ex>[l] \ar@<-1.5ex>[l]}$$
of bounded derived categories of finitely generated modules, writing
$A$ in terms of $B$ and $C$. Then the finiteness of global dimension
(by \cite[Lemma 2.1]{W}) and finitistic dimension (by \cite[3.3]{H})
are invariants. That is, $\gldim (A) <\infty$ if and only
of $\gldim(B)<\infty$ and $\gldim(C) <\infty$, and $\fdim(A)<\infty$
if and only of $\fdim(B)<\infty$ and $\fdim(C)<\infty$.
Denote by $K_0(A)$ the Grothendieck group of $\ma$, which is also
the Grothendieck group of $\dba$. It is a free abelian group with finite rank,
which equals the number of non-isomorphic simple $A$-modules. Given a
recollement as above, there is a decomposition $K_0(\ma) = K_0(\mb) \oplus K_0(\mc)$.

\bigskip


\section{Criteria for the existence of recollements}

Let $A$ be a finite dimensional algebra, and $T$ a compact
exceptional complex over $A$. Let $B=\End_A(T)$ be the endomorphism
algebra. By \cite{Ke}, there exists uniquely a right bounded complex
$X=\widetilde{T}$ of finitely generated projective $B$-$A$-bimodules, such that $X$ as
a complex of $A$-modules is quasi-isomorphic to $T$. This complex
induces a pair of adjoint functors between the unbounded derived
categories of $A$ and $B$:
\[
\xy (-55,0)*{\dua}; {\ar (-27,1.5)*{}; (-43,1.5)*{}_{F}};
{\ar(-43,-1.5)*{}; (-27,-1.5)*{}_{G}}; (-15,0)*{\dub};
\endxy
\]
where $F = -\lten_B X$ and $G = \rhom_A(X,-)$.

\medskip

\begin{lem} With these notations:
\begin{enumerate}

\item The adjoint pair $(F,G)$ restricts to $K^b(P_A)\leftrightarrows K^b(P_B)$
if and only if $X^{tr}:= G(A)$ as a complex of (right) $B$-modules
is compact.
\item The adjoint pair $(F,G)$ restricts to $\dba\leftrightarrows\dbb$ if and only if $X$ as a complex of left $B$-modules is
compact.

\end{enumerate}\label{adjointfun}

\end{lem}

In these cases, the restrictions of $F$ and $G$ are again adjoint to
each other.

\begin{proof} (1) Since $T$ is compact, the derived Hom-functor $G$ is
equivalent to the derived tensor functor $-\lten_A X^{tr}$, where
$X^{tr} = G(A)$. It sends $K^b(P_A)$ to $\tria(X^{tr})$, and it is an
equivalence from $\tria(T)$ to $K^b(P_B)$ with quasi-inverse given by
the restriction of $F$. Hence the functor $F$ sends $K^b(P_B)$ to
$K^b(P_A)$, and the functor $G$ sends $K^b(P_A)$ to $K^b(P_B)$ if and
only if $X^{tr}$ as a complex of $B$-modules is compact.

\medskip

(2) The category $\dba$ is equivalent to the full subcategory of $\dua$
containing those complexes whose cohomology spaces are finite
dimensional. When calculating the cohomology of such a complex, one
may forget its module structure and view it as a chain complex of
vector spaces.

\medskip

We claim that the functor $G$ sends $\dba$ to $\dbb$. For this we
need to show that all simple $A$-modules have images in $\dbb$.
Since $X$ as a complex of right $A$-modules is compact, $X^{tr}$ as
a complex of left $A$-modules is again compact. We take its minimal
projective representative, and thus tensoring with a simple right
$A$-module $S$ would kill all projective modules except the
projective cover of $S$. The compactness of $X^{tr}$ implies the
multiplicity of the corresponding projective cover in the minimal
projective resolution is finite. It follows that the cohomological
space of $G(S) = S\lten_A X^{tr}$ has finite dimension. That is,
$G(S)$ belongs to $\dbb$.

\medskip

Now if $X$ as a complex of left $B$-modules is compact, the same
arguments as above shows $F= -\lten_B X$ sends $\dbb$ to $\dba$.
Conversely, suppose the functor $F$ sends $\dbb$ to $\dba$ and
assume that $X$ is not compact. Then there exists some
indecomposable left $B$-projective module with infinite
multiplicity in $X$. Tensoring its simple top (now as right
$B$-module) with $X$ will provide a complex with infinite
dimensional cohomological space. This contradicts the
assumption of $F$.
\end{proof}

\medskip

\begin{rem}

(1) In general $X$ as a complex of bimodules is not necessarily
compact. For example, take $A$ to be the two dimensional algebra
$k[x]/x^2$, and $T$ to be $A$ itself. So $B$ is identified with
$A$, but $A$ is not compact as $A$-$A$-bimodule.

\medskip

(2) In general $X$ as a complex of left $B$-modules is not
necessarily compact. For example, take $A$ to be the
quasi-hereditary algebra
\[ \xy (-39,0)*{1}; {\ar (-27,1)*{};
(-37,1)*{}_{\alpha}}; {\ar(-37,-1)*{}; (-27,-1)*{}_{\beta}};
(-24,0)*{2};
\endxy
\]
with relation $\beta \circ \alpha = 0$, and $T = P(2)$ the
indecomposable projective module at $2$. Then
$B=\End_A(T) \cong k[x]/x^2$ and $T$ as left $B$-module has infinite
projective dimension.

\end{rem}

\bigskip

The aim of this section is to give a
`finitely generated modules' version of the
criterion for the existence of a recollement given in  \cite[Theorem 1]{K}, \cite[Theorem 2]{NS2}.
We start with a sufficient condition.

\begin{prop} Let $A$ be a finite dimensional algebra. Suppose
there are compact exceptional complexes $\mathcal{C}$ and
$\mathcal{B}$ in $K^b(P_A)$ such that
\begin{enumerate}
\item $\End_A(\mathcal{C})= C$ and $\End_A(\mathcal{B})= B$;
\item $\Hom_A(\mathcal{C},\mathcal{B}[n]) = 0$ for all integers $n$;
\item $\mathcal{C}^{\perp} \cap \mathcal{B}^{\perp} = 0$;
\item the complex of projective $C$-$A$-bimodules $\widetilde{\mathcal{C}}$ and the complex of projective
$B$-$A$-bimodules $\widetilde{\mathcal{B}}$, which are
quasi-isomorphic to $\mathcal{C}$ and $\mathcal{B}$ respectively as
complexes of right $A$-modules, are compact as complexes of left
$C$- and $B$-modules respectively.
\end{enumerate}
Then $A$ admits a recollement of the form
$$\xymatrix@!=5pc{\dbb \ar[r] &\dba \ar@<+1.5ex>[l] \ar@<-1.5ex>[l] \ar[r] &\dbc \ar@<+1.5ex>[l] \ar@<-1.5ex>[l]}$$
for $B=\End_A(\Bcal)$ and $C= \End_A(\Ccal)$. \label{sufficient}
\end{prop}

\begin{proof} By \ref{adjointfun} (2), the
assumption implies the existence of two pairs of adjoint functors
$(i_!, i^!)$ and $(j_!, j^!)$ as in the following partial recollement
$$\xymatrix@!=6pc{\dbb \ar[r]|{i_!} &\dba \ar@<+2ex>[l]|{i^!}
\ar[r]|{j^!} & \dbc \ar@<-2ex>[l]|{j_!} }$$ where $i_!$ and
$j_!$ are full embeddings. This `partial' recollement can be
completed in the same way as in the proof of
\cite[Theorem 1]{K}. We omit the details.
\end{proof}

\medskip

For a discussion of the converse direction, we need more preparations.
The following lemma provides a homological characterization of  compact objects. It has been stated, without proof, in \cite[3.2]{H}.
We include here a proof due to Jiaqun Wei, whom
we thank for suggesting this improvement of our original proof. For an
analogous statement for $D^b(\MA)$ and $K^b(\text{Proj-}A)$ see \cite[Proposition 6.2]{R}.

\medskip

\begin{lem} Let $A$ be a finite dimensional algebra over a field
$k$. Then $X$ in $\dba$ is compact if and only if
for any $Y$ in $\dba$, there exists an integer
$t_0$ with $\Hom(X,Y[t]) = 0$ for all $t\geq t_0$.
\end{lem}

\begin{proof} A compact object is a bounded complex of
finitely generated projective modules. Given such a object $X$, the
condition in the statement is fulfilled. Conversely, suppose $X$
satisfies the assumption. We identify $\dba$ with
$K^{-,b}(P_A)$ and write $X$ as a bounded-above complex of finitely
generated projective modules with bounded homology $\ldots \ra P_{k}
\ra P_{k-1} \ra \ldots \ra P_m \ra 0$. Say $H^i(X) = 0$ for all
$i\leq n$. Set $M = \Coker(P_{n-1} \ra P_n)$, and then $\ldots \ra P_{n-1} \ra P_n \ra 0$ is a projective resolution of $M$. Hence
$\projdim(M)< \infty$ if and only if $X\in K^b(P_A)$. By assumption
there exists $t_0\geq 0$ such that
\begin{enumerate}
\item $\Hom_A(X,S[t]) = 0$ for all simple $A$-module $S$ and for all
$t\geq t_0$;
\item $\Hom_A(X,S[t]) \cong \Hom(M,S[t])$ for all simple $A$-module $S$ and for all
$t\geq t_0$.
\end{enumerate}
Hence $\Ext^t_A(M,S)=0$ for all simple $A$-module $S$ and for all
$t\geq t_0$. Namely $\projdim(M)\leq t_0$.
\end{proof}

\medskip

\begin{cor} A triangulated functor $F: \dba \ra \dbb$ restricts to $K^b(P_A) \ra K^b(P_B)$ provided it has a right
adjoint functor. \label{perfect}
\end{cor}

\medskip

Now we can show that some of the conditions in Proposition \ref{sufficient}
are also necessary.

\begin{cor} Suppose $A$ admits a recollement of the form
$$\xymatrix@!=5pc{\dbb \ar[r] &\dba \ar@<+1.5ex>[l] \ar@<-1.5ex>[l] \ar[r] &\dbc \ar@<+1.5ex>[l] \ar@<-1.5ex>[l]}$$
with finite dimensional algebras $B$ and $C$. Then  $\mathcal{C} = j_!(C)$ and $\Bcal = i_*(B)$
are compact exceptional  objects  satisfying
\begin{enumerate}

\item $\End_A(\mathcal{C})= C$ and $\End_A(\mathcal{B})= B$;

\item $\Hom_A(\mathcal{C},\mathcal{B}[n]) = 0$ for all integers $n$;

\item $\tria (\Bcal \oplus \Ccal) = \dba$.
\end{enumerate}  \label{necessary}
\end{cor}
\begin{proof}
Corollary \ref{perfect} implies the
compactness of $\Ccal$ and $\Bcal$. Since
$j_!$ and $i_*$ are full embeddings, $\Bcal$ and $\mathcal{C}$ are exceptional and
condition (1) is satisfied. Conditions (2) and (3) follow directly from the
definition of recollement. \end{proof}

\medskip

There is still  an obstruction: the adjoint pairs $(i_!, i^!)$ and $(j_!, j^!)$ are
in general not necessarily the derived tensor- or hom-functors
induced by $i_*(B)$ and $j_!(C)$. This is the case, however, when
the algebra $A$ has finite global dimension. The next result asserts
that, up to equivalence of recollements of $\dba$.

\medskip

\begin{thm} Let $A$ be a finite dimensional algebra with finite global dimension. The following
statements are equivalent.

\begin{enumerate}

\item $A$ admits a recollement of the form
$$\xymatrix@!=5pc{\dbb \ar[r] &\dba \ar@<+1.5ex>[l] \ar@<-1.5ex>[l] \ar[r] &\dbc \ar@<+1.5ex>[l] \ar@<-1.5ex>[l]}$$
with finite dimensional algebras $B$ and $C$.

\item There exist exceptional complexes $\mathcal{C}$ and
$\mathcal{B}$ in $\dba)$ such that

\begin{enumerate}

\item $\End_A(\mathcal{C})= C$ and $\End_A(\mathcal{B})= B$;

\item $\Hom_A(\mathcal{C},\mathcal{B}[n]) = 0$ for all integers $n$;

\item $\tria (\Bcal \oplus \Ccal) = \dba$.






\end{enumerate}

\end{enumerate} \label{crite-finite}

\end{thm}

\begin{proof}
(1) $\Rightarrow$ (2) is Corollary \ref{necessary}.

(2) $\Rightarrow$ (1): By \cite[Theorem 1]{K}, \cite[Theorem
2]{NS2}, there is a recollement
$$\xymatrix@!=5pc{D^-(\MB) \ar[r] & D^-(\MA) \ar@<+1.5ex>[l] \ar@<-1.5ex>[l] \ar[r] & D^-(\MC) \ar@<+1.5ex>[l] \ar@<-1.5ex>[l]}$$
with $B=\End_A(\Bcal)$ and $C=\End_A(\Ccal)$. It follows then from
\cite[Corollary 5]{K} that $B$ and $C$ have finite global
dimension since $A$ has so. By Keller's construction \cite{Ke}, the
bicomplexes $\widetilde{\Bcal}$ and $\widetilde{\Ccal}$ are right
bounded and have bounded cohomologies. Hence they are compact as
left $B$- and $C$-complexes respectively. That is, condition (4)
in Proposition \ref{sufficient} is satisfied.
\end{proof}

\medskip

As a corollary there is the following assertion on lifting and restricting recollements.

\begin{cor} Let $A,B,C$ be  finite dimensional algebras.

\smallskip

(1) Any recollement of bounded derived categories
$$\xymatrix@!=5pc{\dbb \ar[r] &\dba \ar@<+1.5ex>[l] \ar@<-1.5ex>[l] \ar[r] &\dbc \ar@<+1.5ex>[l] \ar@<-1.5ex>[l]}$$
can be lifted to a recollement of unbounded derived categories
$$\xymatrix@!=5pc{D(\MB) \ar[r] & D(\MA) \ar@<+1.5ex>[l] \ar@<-1.5ex>[l] \ar[r] & D(\MC) \ar@<+1.5ex>[l] \ar@<-1.5ex>[l]}$$

\smallskip

(2) If $A$ has finite
global dimension,  any recollement of unbounded derived categories
$$\xymatrix@!=5pc{D(\MB) \ar[r] & D(\MA) \ar@<+1.5ex>[l] \ar@<-1.5ex>[l] \ar[r] & D(\MC) \ar@<+1.5ex>[l] \ar@<-1.5ex>[l]}$$
can be restricted to a recollement of bounded derived categories
$$\xymatrix@!=5pc{\dbb \ar[r] &\dba \ar@<+1.5ex>[l] \ar@<-1.5ex>[l] \ar[r] &\dbc \ar@<+1.5ex>[l] \ar@<-1.5ex>[l]}$$
\label{finitrestriction}
\end{cor}

\begin{proof}
(1) By Corollary \ref{necessary}, we have a pair $(j_!(C), i_*(B))$ of compact exceptional
objects, which yields the existence of a
$D(\rmod)$ level recollement by \cite[5.2.9]{NZ} (or \cite[Theorem
2.2]{AKL2}).

(2)
The object $\Ccal=j_!(C)$ is always  compact by \cite[5.2.9]{NZ} (or \cite[Theorem
2.2]{AKL2}), and we show in \cite{AKLY} that $\Bcal=i_*(B)$ is compact whenever $A$ has finite
global dimension. To apply Theorem \ref{crite-finite}, it only remains  to check $A \in
\tria(\Bcal\oplus\Ccal)$. This is true because
$\Tria(\Bcal\oplus\Ccal) = \dua$ and $\Bcal\oplus\Ccal$ is compact.
\end{proof}

\bigskip


\section{Constructing recollements from tilting modules}

Let $A$ be a finite dimensional algebra over a field $k$, and $T$
a tilting module (see \ref{tilting}). Note that
the $T$-resolution of $A$, $$(*)\ \ \ \ \ \ \ 0 \ra A \ra T_0 \ra
T_1 \ra 0$$ is not required to be minimal ($T_i \in \add(T)$). We fix
$T$ together
with such a $T$-resolution of $A$. 

Our aim is to give an analogue of \cite[Theorem 4.8]{AKL}.

\medskip

For an $A$-module $X$, the (module) {\em perpendicular category}
of $X$, denoted by $\widehat{X}$, is by definition the full
subcategory of $\ma$, consisting of the modules $X$ such that
$\Hom_A(X,M)=0=\Ext^1_A(X,M)$.


\begin{lem}[cf.~\cite{CTT}, Proposition 1.3] The perpendicular category $\widehat{T_1}$ of
$T_1$ is a reflective
subcategory of $\ma$. In other words, the full embedding $i:
\widehat{T_1} \ra \ma$ admits a left adjoint functor $\ell$. \label{adjoint}
\end{lem}
\begin{proof}
In \cite{CTT}, this statement has been proved for Mod-$A$ by giving an explicit construction of  the left adjoint functor $\ell$.  Since  $\ell$ restricts to finite dimensional modules,
the same argument works for $\ma$. For the reader's convenience, we recall the construction.

\smallskip

It consists of  two steps: given an $A$-module $M$, first the universal
extension of $T_1$ with respect to $M$ is formed, that is, a short exact
sequence $0 \ra M \ra M' \ra T_1^n \ra 0$, for some natural number
$n$, such that any extension of $\Ext^1_A(T_1,M)$ is a pullback
along a map $T_1 \ra T_1^n$. Secondly, we factor out the trace of
$T_1$ in $M'$. The factor module provides
exactly the image $\ell(M)$ of $M$. In other
words, the composition $M\ra M' \ra \ell(M)$ is the left approximation
of $M$ in $\widehat{T_1}$ with $\Hom(M'',N) \xrightarrow{\thicksim}
\Hom(M,N)$ for any $N\in \widehat{T_1}$.

\smallskip

Notice that the second step is not required when the endomorphism ring of $T_1$ is a skew field, for in this case we can choose $n={\rm dim}_{\End_A T_1}\Ext^1_A(T_1,M)$ to obtain $M'\in \widehat{T_1}$, cf. \cite[Appendix A.1]{AKL}.
\end{proof}

\medskip

We now compute $\ell(A)$.
The $T$-resolution $(*)$ of $A$ is a universal extension
of $T_1$ with respect to $A$. Indeed, applying $\Hom_A(T_1,-)$ we
get a surjection $\Hom_A(T_1,T_1) \ra \Ext^1_A(T_1,A)$. It follows
that the left approximation of $A$ in $\widehat{T_1}$ is
$$\ell(A)=T_0/\tau_{T_1}(T_0)$$ where $\tau_{T_1}(T_0)$ is the trace of $T_1$
in $T_0$. We  write $B$ for the endomorphism algebra of $\ell(A)$.

\medskip

\begin{lem}[\cite{GL}, Proposition 3.8] Notations are as above. The module $\ell(A)$ is a projective generator of $\widehat{T_1}$.
It determines a natural algebra homomorphism $\varphi: A\ra B$ which
is a ring epimorphism such that the image of the full embedding
$\varphi_*: \mb \ra \ma$ is equivalent to $\widehat{T_1}$.




\end{lem}

The proof is by checking directly that $\Hom_A(\ell(A),M) \cong
\Hom_A(A,M) \cong  M$ for any $M\in \widehat{T_1}$. In particular
$B = \Hom_A(\ell(A),\ell(A)) \cong \ell(A)$. This is actually an isomorphism
of $B$-$A$-bimodules (where $B$ is equipped with an $A$-bimodule
structure via $\varphi$).
By \cite[1.7]{AA} the ring epimorphism $\varphi$ can be identified with the universal localisation $A_{T_1}$ of $A$ at $T_1$,
see also \cite[4.1]{AKL}.

\bigskip When does there exist a recollement of the form
$$\xymatrix@!=5pc{\dbb \ar[r] &\dba \ar@<+1.5ex>[l] \ar@<-1.5ex>[l] \ar[r] &\dbc \ar@<+1.5ex>[l] \ar@<-1.5ex>[l]}$$
for some finite dimensional algebra $C$?


\begin{thm} Suppose the algebra homomorphism $\varphi:A\ra B$
is a homological epimorphism, and the projective dimension of $T_1$
as a left $C:=\End_A(T_1)$-module is finite. Then there is the
following recollement of $\dba$

$$\xymatrix@!=6pc{\dbb \ar[r]|{i_*=i_!} &\dba \ar@<+2.5ex>[l]|{i^!}
\ar@<-2.5ex>[l]|{i^*} \ar[r]|{j^!=j^*} & \dbc\ar@<+2.5ex>[l]|{j_*}
\ar@<-2.5ex>[l]|{j_!} }$$ where $i^*=-\lten_A B$, $i_*=
\varphi_*$, $i^!=\rhom_A(B,-)$, $j_!=-\lten_C T_1$, and
$j^!=\rhom_A(T_1,-)$. \label{tiltreco}
\end{thm}

\begin{proof} We need to make a detour through the unbounded derived category $D(\MA)$.
Combining Example 4.5 and Theorem 4.8 in \cite{AKL}, we obtain a
recollement $$\xymatrix@!=5pc{\dub \ar[r] & \dua \ar@<+1.5ex>[l]
\ar@<-1.5ex>[l] \ar[r] & \duc \ar@<+1.5ex>[l] \ar@<-1.5ex>[l]}$$
where the functors are as required and $j_!(C)$ is isomorphic to
$T_1$. For the convenience of the reader, we include more details
here.

\medskip

The module $T_1$ is finite dimensional of projective dimension
$\le 1$, so it is compact. Hence it generates a smashing
 subcategory $\mcx = \Tria(T_1)$ of $\dua$
(\cite[4.5]{AJS1}). It follows that there is a recollement
$$\xymatrix@!=4pc{\mcy \ar[r] &\dua \ar@<+1.5ex>[l]
\ar@<-1.5ex>[l] \ar[r] & \Tria(T_1) \ar@<+1.5ex>[l]
\ar@<-1.5ex>[l] }$$ for $\mcy=T_1^{\bot}$ (\cite[4.4.14]{NZ}). 

\medskip
By \cite{AA}, the universal localisation of the ring $A$ at $T_1$
is given by  $\varphi:A\ra B$. Because $\varphi$ is a homological
epimorphism, by \cite{NR} (more precisely, see \cite[Theorem
1.8]{AKL}), the recollement induced by $\varphi$ (see
\ref{homoepi}) has the following form
$$\xymatrix@!=4.5pc{\dub \ar[r] &\dua \ar@<+1.5ex>[l]
\ar@<-1.5ex>[l] \ar[r] & \Tria(T_1) \ar@<+1.5ex>[l]
\ar@<-1.5ex>[l] }.$$  The two recollements
are equivalent. In particular, $\mcy \cong \dub$,
and the functors on the left hand side are as required.
On the other hand, $T_1$ is an exceptional compact generator of
$\Tria(T_1)$, and by Rickard's or by Keller's Morita Theorem, $\Tria(T_1)$ is, as a
triangulated category, equivalent to the derived category $\duc$. So
the desired recollement of $\dua$ has been obtained, with $j_!$ and $j^!$ as
required and $j_!(C)=T_1$.

\medskip

Now set $\Bcal = \ell(A)\ (\cong i_*(B))$ and $\Ccal = T_1\ (= j_!(C))$.
Since $\varphi: A\ra B$ is a homological epimorphism, there is a
full embedding $\varphi_*:\dbb \ra \dba$ which admits a right
adjoint functor (see Subsection 1.3). By Corollary \ref{perfect},
$(\ell(A))_A \cong B_A = \varphi_*(B_B)$ is compact. It is also
exceptional as $B_B$ is so. By definition $T_1\in\add(T)$ is
compact and exceptional, too. Moreover,  the conditions (1)-(4) of
Proposition \ref{sufficient} hold: (1) follows by construction, (2)
is implied by $\ell(A) \in \widehat{T_1}$ and $\projdim(T_1)\le 1$, (3) is a consequence
of the $D(\rmod)$ recollement above, and (4) follows from $\ell(A)
\cong B$ as left $B$-modules and from the assumption on $C$.
\end{proof}

\medskip

When $A$ has finite global dimension, there is  the
following simplified  version.

\begin{prop} Assume that $A$ has finite global dimension. If the module
$\ell(A)$ is exceptional, then the algebra homomorphism $\varphi:A \ra
B$ as in 3.2 is a  homological epimorphism, and there is the
following recollement of $\dba$
$$\xymatrix@!=5pc{\dbb \ar[r] & \dba \ar@<+1.5ex>[l]
\ar@<-1.5ex>[l] \ar[r] & \dbc \ar@<+1.5ex>[l] \ar@<-1.5ex>[l]}$$
where the functors $i^*$, $i_*$, $i^!$, $j_!$ and $j^!$ are as in
the Theorem \ref{tiltreco}. \label{finitiltreco}

\end{prop}

\begin{proof}

Since $\ell(A)$ is exceptional, the map $\varphi: A \ra B$ is a homological
epimorphism, see \cite[4.9]{GL}. Thus there is a $D(\rmod)$ level recollement
$$\xymatrix@!=5pc{\dub \ar[r] & \dua \ar@<+1.5ex>[l]
\ar@<-1.5ex>[l] \ar[r] & \duc \ar@<+1.5ex>[l] \ar@<-1.5ex>[l]}$$
 as in the proof of
Theorem \ref{tiltreco}, and the statement follows from Corollary \ref{finitrestriction}.
\end{proof}




























Note that a Theorem of Happel \cite[3.3]{H5} is a special case when
the endomorphism ring of $\ell(A)$ is the base field $k$. The proof
there uses the criterion of \cite{K} which in fact has been
stated  for big module categories.

\medskip

When the trace of $T_1$ in $T_0$ is trivial, $\ell(A)$
coincides with $T_0$ and hence it is exceptional.

\begin{cor} Assume that $A$ has finite global dimension. If there
is no nonzero homomorphism from $T_1$  to $T_0$, then there is the
following recollement
$$\xymatrix@!=5pc{\dbb \ar[r] & \dba \ar@<+1.5ex>[l]
\ar@<-1.5ex>[l] \ar[r] & \dbc \ar@<+1.5ex>[l] \ar@<-1.5ex>[l]}$$
where $B=\End_A(T_0)$ and $C=\End_A(T_1)$.
\end{cor}

In particular, $\varphi:A\ra B$ is an injective homological
epimorphism. (See \cite{AHS} for more information on injective ring epimorphisms.
This setup has been one of the motivations for the current work.)
By \cite[Proposition 4.13]{GL}, there even exists a homological
epimorphism $\psi: A\ra C:=\End_A(T_1)$. Let $\psi_*: \dbc\ra\dba$
be the induced full embedding. Then $\psi_*\circ[1]$ provided the
functor $j_*$ in the recollement.

\medskip

For example, take an indecomposable exceptional module $M$
satisfying $\Hom_A(M,A)=0$ and $\End_A(M)=k$. The
Bongartz complement of $M$ is a universal
extension $0 \ra A \ra M' \ra M^{\oplus n} \ra 0$ where $n=\dim_k
\Ext^1_A(M,A)$. Then $M\oplus M'$ is a tilting module
and $\Hom_A(M,M')=0$.

\bigskip

\section{Examples}
In the following some examples are given of constructing recollements from a
tilting module. In particular we will see that the assumptions of
Theorem \ref{tiltreco} are optimal. As in
the previous section, $A$ is a finite dimensional $k$-algebra, $T$
is a tilting $A$-module with a $T$-resolution of $A$: $0 \ra A \ra T_0 \ra T_1$ (where $T_0, T_1 \in \add(T)$),
$\ell(A)=T_0/\tau_{T_1}(T_0)$ is the left approximation of $A$ in
$\widehat{T_1}:=\{M\in \ma: \Hom_A(T_1,M)=0=\Ext^1_A(T_1,M)\}$,
$B:=\End_A(\ell(A))$ ($\cong \ell(A)$ as right $A$-module),
$C:=\End_A(T_1)$, and $\varphi: A \ra B$ is a ring epimorphism with
$\mb \cong \widehat{T_1}$.

\medskip

\begin{ex} In general, the ring epimorphism $\varphi:A\ra B$
need not be a homological epimorphism.

\end{ex}

Let $A$ be the path algebra of the quiver
\[ \xy (-39,0)*{1}; {\ar (-27,1)*{};
(-37,1)*{}_{\alpha}}; {\ar(-37,-1)*{}; (-27,-1)*{}_{\beta}};
(-24,0)*{2};  {\ar (-12,1)*{}; (-22,1)*{}_{\gamma}};
{\ar(-22,-1)*{}; (-12,-1)*{}_{\delta}}; (-9,0)*{3};
\endxy
\]
with relations $\alpha\circ\gamma = 0$, $\delta \circ\beta = 0$,
$\beta\circ\alpha = \gamma \circ\delta$ and $\delta\circ\gamma = 0$.
This is the same example as in \cite[1.5]{H5}. The
indecomposable projective $A$-modules are
$$P_1 =\ \begin{matrix} 1 \\ 2 \\ 1 \end{matrix}\ ,\ \ \ \
P_2 =\ \begin{matrix} & 2 & \\ 1 & & 3 \\ & 2 \end{matrix}\ ,\ \ \ \
P_3 =\ \begin{matrix} 3 \\ 2\\ \end{matrix}\ .$$ Take $T=P_1\oplus
P_2\oplus T_1$ where $T_1 = ~\begin{matrix} 2 \\ 1 \end{matrix}~$ is
the quotient of $P_2$ factoring out $P_3$. It is clear that $T$ is a
tilting module and
$$0\ra A= P_1\oplus P_2\oplus P_3 \ra P_1\oplus P_2\oplus P_2 \ra T_1 \ra 0 $$
is a $T$-resolution of $A$. The trace of $T_1$ in $P_1$ and $P_2$ is
isomorphic to $T_1$ and $2$ respectively. So the left approximation
of $A$ in $\widehat{T_1}$ is $\ell(A) = T_0/\tau_{T_1}(T_0) \cong
1\oplus \left(\begin{matrix} & 2 & \\ 1 & & 3 \end{matrix}\right)
\oplus \left(\begin{matrix} & 2 & \\ 1 & & 3 \end{matrix} \right)$.
Note that $\Ext_A^2(1, ~\begin{matrix} & 2 & \\ 1 & & 3
\end{matrix}~) \neq 0$. So, $\ell(A)$ as an $A$-module is not
exceptional, and hence $\varphi$ cannot be a homological
epimorphism.

\medskip

\begin{ex} In general,
$T_1$  as a left $C$-module may have infinite projective dimension.

\end{ex}

This is an example from \cite{K}. Let $A$ be the path algebra of the
quiver
\[ \xy (-39,0)*{1}; {\ar (-27,1)*{};
(-37,1)*{}_{\alpha}}; {\ar(-37,-1)*{}; (-27,-1)*{}_{\beta}};
(-24,0)*{2}; \endxy \] with relation $\alpha\circ\beta\circ\alpha =
0$. So the indecomposable projective $A$-modules are
\[P_1 =\ \begin{matrix} 1\\2\\1\\2\end{matrix}\ , \ \ \ \
P_2=\ \begin{matrix} 2\\1\\2\end{matrix}\ .\] Take $T=A = P_1\oplus
P_2$ the regular module and
$$0\ra A = P_1\oplus P_2 \ra T_0 = P_1 \oplus P_2 \oplus P_2 \ra T_1 = P_2 \ra 0$$
as $T$-resolution of $A$ (not minimal). Then $\ell(A) = T_0
/\tau_{T_1}(T_0) = S_1$. It has no self-extension. In fact,
$B=\End_A(\ell(A)) \cong k$ and $\varphi:A\ra B$ is a homological
epimorphism. But $C=\End_A(T_1) \cong k[x]/(x^2)$ and $T_1$ as a
left $C$-module is isomorphic to $C\oplus k$, where $k$ is the
simple $C$-module with infinite projective dimension.

\medskip

\begin{ex} Here the conditions of Theorem \ref{tiltreco} are satisfied.
\label{positivexample}
\end{ex}

Let $A$ be the path algebra of the quiver
\[ \xy (-39,0)*{1}; {\ar (-27,1)*{};
(-37,1)*{}_{\alpha}}; {\ar(-37,-1)*{}; (-27,-1)*{}_{\beta}};
(-24,0)*{2}; \endxy \] with relation $\beta\circ\alpha =0$. So the
indecomposable projective $A$-modules are \[P_1=\ \begin{matrix}1
\\2\end{matrix}\ ,\ \ \ \  P_2=\ \begin{matrix} 2\\1\\2\end{matrix}\
.\] The global dimension of $A$ is $2$. Take $T= P_2 \oplus S_2$. It
is a tilting module with the following resolution of $A$:
$$0\ra A=P_1\oplus P_2 \ra T_0 = P_2\oplus P_2 \ra S_2 \ra 0.$$ Then
$\ell(A) = T_0/\tau_{T_1}(T_0) =
\begin{matrix} 2\\1\end{matrix} \oplus \begin{matrix}
2\\1\end{matrix}~$. It has no self-extension. By Proposition
\ref{finitiltreco}, the ring epimorphism $\varphi:A\ra B\cong
M_2(k)$ is homological. Indeed it sends $e_i$ to $E_{ii}$ (for
$i=1,2$), $\alpha$ to $E_{21}$ and $\beta$ to $0$. On the other
hand, $C=\End_A(T_1) \cong k$ and $T_1 \cong k$ is projective as
$C$-module. So there is a recollement
$$\xymatrix@!=5pc{\dbb \ar[r] & \dba \ar@<+1.5ex>[l]
\ar@<-1.5ex>[l] \ar[r] & \dbc \ar@<+1.5ex>[l] \ar@<-1.5ex>[l]}$$
with $\dbb \cong D^b(\text{mod-}k) \cong \dbc$.

\medskip

\begin{ex} The standard stratification of
quasi-hereditary algebras. \label{quasihereditary}
\end{ex}

Recall \cite{CPS} that a two-sided ideal $J$ of a finite dimensional algebra
$A$ is a {\em heredity ideal}, if $J=AeA$ is generated by some
idempotent $e$ and $J$ is projective as $A$-module and $eAe$ is a
semi-simple algebra. The algebra $A$ is called {\em
quasi-hereditary}, if there exists a chain $0 = J_0 \subset J_1
\subset \ldots \subset J_s=A$ of two-sided ideals of $A$, such that
$J_i/J_{i-1}$ is a heredity ideal of $A/J_{i-1}$ for all $i\geq 1$.
Such a chain is called a {\em heredity chain} of $A$ (not
necessarily unique). By Parshall and Scott \cite[Theorem 2.7(b)]{PS}, an ideal $J=AeA$ appearing in a heredity chain induces a recollement of the form
$$\xymatrix@!=6pc{D^b(\text{mod-$A/AeA$}) \ar[r] & \dba \ar@<+1.5ex>[l]
\ar@<-1.5ex>[l] \ar[r] & D^b(\text{mod-$eAe$}) \ar@<+1.5ex>[l]
\ar@<-1.5ex>[l]}.$$ This fits in our setup.

\medskip

Let $A$ be a quasi-hereditary algebra. In particular it has finite
global dimension. Let $e=e^2$ be an idempotent in $A$ such that
$J=AeA$ is an ideal in a heredity chain of $A$. Take $T=A$ with
$$0\ra A \ra T_0 = A\oplus eA \ra T_1= eA \ra 0$$ the $T$-resolution
of $A$ (not minimal). So $\ell(A) = T_0/\tau_{T_1}(T_0) = A/AeA$. Note
that $eAe$ and $A/AeA$ are again quasi-hereditary algebras. Hence
$A/AeA$ is exceptional and $\varphi:A\ra A/AeA$ is a homological
epimorphism. Proposition \ref{finitiltreco} reasserts the existence
of the standard recollement. In this case all the functors can be
written down explicitly: $i^* = -\lten_A A/AeA$, $i^! =
\rhom_A(A/AeA,-)$, $j_! = \lten_{eAe} eA$ and
$j_*=\rhom_{eAe}(Ae,-)$.

\medskip

\begin{rem} The algebra $A$ in Example \ref{positivexample} is
quasi-hereditary with a heredity chain $0\subset Ae_1A\subset A$.
However, the induced standard recollement is not equivalent to the
recollement in Example \ref{positivexample}. This shows the
recollements obtained from a tilting module depends on the choice of
the resolution of $A$.
\end{rem}

\medskip

\begin{ex} An injective homological epimorphism.
\end{ex}

Let $A$ be the path algebra of $1\xrightarrow{\alpha} 2
\xrightarrow{\beta} 3$. So $A$ is hereditary with indecomposable
projective modules $$P_1=\ 1,\ \ \ \ P_2=\ \begin{matrix} 2\\1
\end{matrix}\ , \ \ \ \ P_3=\ \begin{matrix} 3\\2\\1\end{matrix}\
.$$ Take $T= P_1\oplus P_3\oplus S_3$. It is a tilting module with a
resolution of $A$: $$0\ra A=P_1\oplus P_2\oplus P_3 \ra
T_0=P_1\oplus P_3\oplus P_3 \ra T_1= S_3 \ra 0.$$ Clearly, $T_1$
does not map non-trivially to $T_0$. The endomorphism ring of
$\ell(A)=T_0$ is the path algebra of
\[ \xy (-39,0)*{1}; {\ar(-37,0)*{}; (-27,0)*{}_{\alpha}};
(-24,0)*{2};  {\ar (-12,-1)*{}; (-22,-1)*{}^{\gamma}};
{\ar(-22,1)*{}; (-12,1)*{}^{\beta}}; (-9,0)*{3};
\endxy
\]
with relations $\beta\circ\gamma = e_3$ and $\gamma\circ\beta =
e_2$. Then $A$ embeds into $B$, which is
Morita equivalent to the path algebra of $1\rightarrow 2$.

\bigskip

\section{Hereditary and piecewise hereditary algebras}

In the first subsection, the focus will be on hereditary algebras, i.e.
algebras of global dimension one. In the second subsection, the remaining
case of weighted projective lines will be considered. Combining the
results will yield a Jordan H\"older theorem both in the small world
of bounded derived categories of finitely generated modules and in
the large world of unbounded derived categories of (possibly infinitely
generated) modules.

\medskip

\subsection{Hereditary algebras}

The first result states that any recollement of a finite dimensional
hereditary algebra, bounded or unbounded, is uniquely determined by
the same datum, namely
a compact and exceptional object. It is inspired by
\cite[Proposition 3]{HRS} and \cite[Theorem 2.5 and Corollary
3.3]{AKL2}. Bijections between homological epimorphisms  and recollements
as well as various other bijections have already been established by
Krause and Stovicek \cite[Theorem 8.1]{KS}, in a different way.
Throughout this section $k$ is an arbitrary field.

\medskip

\begin{thm} Let $A$ be a finite dimensional hereditary algebra over a field $k$.

There are one to one correspondences between the equivalence
classes of the following: \label{propbij}

\begin{enumerate}

\item Exceptional objects in $\dba$.

\item Homological epimorphisms $A \ra B$, where $B$ is a  finite dimensional algebra.

\item Recollements of the form
$$\xymatrix@!=5pc{\dbb \ar[r] & \dba \ar@<+1.5ex>[l]
\ar@<-1.5ex>[l] \ar[r] & \dbc \ar@<+1.5ex>[l] \ar@<-1.5ex>[l]}$$

\item Recollements of the form
$$\xymatrix@!=5pc{D^b(\MB) \ar[r] & D^b(\MA) \ar@<+1.5ex>[l]
\ar@<-1.5ex>[l] \ar[r] & D^b(\MC) \ar@<+1.5ex>[l] \ar@<-1.5ex>[l]}$$

\item Recollements of the form
$$\xymatrix@!=5pc{D^-(\MB) \ar[r] & D^-(\MA) \ar@<+1.5ex>[l]
\ar@<-1.5ex>[l] \ar[r] & D^-(\MC) \ar@<+1.5ex>[l] \ar@<-1.5ex>[l]}$$

\item Recollements of the form
$$\xymatrix@!=5pc{\dub \ar[r] & \dua \ar@<+1.5ex>[l]
\ar@<-1.5ex>[l] \ar[r] & \duc \ar@<+1.5ex>[l] \ar@<-1.5ex>[l]}$$
\end{enumerate} \label{hered}
where $B$ and $C$ in {\rm (3) to (6)} are finite dimensional algebras.

\end{thm}

Here two exceptional objects $X$ and $Y$ are said to be {\em
equivalent}, if they generate the same triangulated category, i.e.
$\tria (X)= \tria (Y)$. Two homological epimorphisms $\varphi: A\ra B$
and $\varphi':A\ra B'$ are {\em equivalent}, when the essential
images of the full embeddings $\varphi_*:\dbb\ra\dba$ and
$\varphi'_*:\dbb\ra\dba$ coincide. The equivalence of two
recollements has been defined in Section 1.1.

\begin{proof} (1) $\Rightarrow$ (2), (3): Let $X$ be an exceptional complex in $\dba$.
The proof of \cite[Corollary 3.3]{AKL2} carries over to produce
recollements on $D^b(\rfmod)$ level. In particular, the
recollement in (3) is induced by the homological epimorphism in (2)
with $C=\End_A(X)$, and the essential image of $j_!$ is $\tria (X)$.

\medskip

(2) $\Rightarrow$ (1): Starting from a homological epimorphism
$\varphi: A\ra B$, we would like to get an exceptional object
$X\in\dba$ such that the essential image of $\varphi_*:\dbb\ra\dba$
equals the right perpendicular category $X^\perp$ of $X$, or
equivalently $\tria X$ equals the left perpendicular category of
$\varphi_*(B)$. Serre duality is well-known to hold in
$\dba$, that is $\Hom_A(M,N) \cong D\Hom_A(N,SM)$ for
$M,N\in\dba$, where $D=\Hom_k(-,k)$, $S=\nu$ the Nakayama functor,
and also $S=\tau\circ[1]$ where $\tau$ is the Auslander-Reiten
translation. Therefore the left perpendicular category of
$\varphi_*(B)$ coincides with the right perpendicular category of
$S^{-1}\circ\varphi_*(B)$. It is clear that $\varphi_*(B)$ is a
partial tilting $A$-module. Since $S$ is an autoequivalence of
$\dba$, $S^{-1}\circ\varphi_*(B)$ is exceptional in $\dba$. Now
apply (1) $\Rightarrow$ (3) to $S^{-1}\circ\varphi_*(B)$. Note that
$\End_A(S^{-1}(\varphi_*(B))) \cong \End_A(\varphi_*(B))
\cong \End_B(B) \cong B$. We obtain a homological ring
epimorphism $A\ra B'$ as well as the induced recollement
$$\xymatrix@!=5pc{D^b(\text{mod-}B') \ar[r] & \dba \ar@<+1.5ex>[l]
\ar@<-1.5ex>[l] \ar[r] & \dbb \ar@<+1.5ex>[l] \ar@<-1.5ex>[l]}$$
with $\tria (i_*(B')) = \Img(i_*) = (S^{-1}\circ\varphi_*(B))^\perp$.
So $i_*(B')$, i.e. $B'$ viewed as an $A$-module via the homological
epimorphism $A\ra B'$, is the exceptional object we are looking for.

\medskip

(4) $\Leftrightarrow$ (5) follows from \cite[Proposition 4 and
Corollary 6]{K}, see also \cite[Lemma 4.1]{AKL2}.

\medskip

(3) $\Rightarrow$ (5): Given a recollement as in (3), the objects $j_!(C)$ and
$i_*(B)$ guarantee the existence of a
recollement of the form (5), by the characterisations in \cite[Theorem 1]{K} and
\cite[Theorem 2]{NS2}.

\medskip

(5) $\Rightarrow$ (6) follows from \cite[Lemma 4.3]{AKL2}.

\medskip

(6) $\Rightarrow$ (1): Given a $D(\rmod)$ level recollement, we get
back a compact and exceptional object $j_!(C)$, following
\cite[5.2.9]{NZ}, \cite[Theorem 2.2]{AKL2}.
\end{proof}

\medskip

\begin{rem} (1) In the proof of (2) $\Rightarrow $ (1), we have obtained a
recollement
$$\xymatrix@!=5pc{D^b(\text{mod-}B') \ar[r] & \dba \ar@<+1.5ex>[l]
\ar@<-1.5ex>[l] \ar[r] & \dbb \ar@<+1.5ex>[l] \ar@<-1.5ex>[l]}.$$
Comparing with the recollement in (3), it is not difficult to see
that $B'$ is derived equivalent to $C$. It is a general phenomenon
for algebras of finite global dimension, where Serre duality holds,
that the two sides of a recollement can be switched.

\medskip

(2) The following fact is also implicit in the proof: given an
exceptional object $X$ in $\dba$, there exists a homological
epimorphism $\varphi:A\ra B$ such that the essential image of the
full embedding $\varphi_*:\dbb\ra\dba$ is $\tria X$. In other
words $B$ is derived equivalent to $\End_A(X)$. So in a
recollement of $\dba$, there are three homological epimorphisms
hidden, corresponding to $j_!(C)$, $i_*(B)$ and $j_*(C)$
respectively.\\

\end{rem}

\medskip


\begin{cor} Let $A$ be a finite dimensional hereditary algebra with $n$ nonisomorphic simple
modules. Let $T\in \dba$ be multiplicity-free  and exceptional.
The following assertions are equivalent:

\begin{enumerate}

\item $T$ is a tilting complex;

\item The number of indecomposable direct summands of $T$ equals $n$;

\item The perpendicular category $\tria(T)^\bot$ in $\dba$ vanishes.

\end{enumerate}
\label{number}
\end{cor}

\begin{proof} (1) $\Rightarrow$ (2): If $T$ is tilting, then $A$
is derived equivalent to the endomorphism algebra $B$ of $T$. So
they have the same number of non-isomorphic simple modules. This
number of $B$ equals the number of indecomposable direct summands
of $T$.

(2) $\Rightarrow$ (3): By Theorem \ref{propbij}, $T$ generates
a recollement
$$\xymatrix@!=5pc{\dbb \ar[r] & \dba \ar@<+1.5ex>[l]
\ar@<-1.5ex>[l] \ar[r] & \dbc \ar@<+1.5ex>[l] \ar@<-1.5ex>[l]}$$
where $C=\End_A(T)$. It follows then from  Proposition
\ref{Grothendieck} that $K_0(B)=0$. Hence $\tria(T)^\bot \cong
\dbb$ must be trivial.

(3) $\Rightarrow$ (1) is straightforward from Theorem \ref{propbij}.

\end{proof}

\medskip

Combining Proposition \ref{hered} and \cite[Theorem 6.1]{AKL2}, we
obtain the derived Jordan-H\"older theorem for hereditary
algebras on $D^b(\rfmod)$, $D^b(\text{Mod}-)$ and $D^-(\text{Mod}-)$
levels, as well as on $D(\text{Mod}-)$ level.

\begin{cor} Let $A$ be a finite dimensional hereditary algebra
and let $S_1, \ldots, S_n$ be the representatives of isomorphism
classes of simple $A$-modules. Denote by $D_i$ the endomorphism
rings of $S_i$ ($1\le i \le n$). Then $\dba$ ($D^b(\MA)$,
$D^{-}(\MA)$) has a stratification with $D^b(\text{mod-}D_i)$
($D^b(\text{Mod-}D_i)$, $D^{-}(\text{Mod-}D_i)$ respectively) ($1\le
i \le n$) being the factors. Moreover, any stratification of $\dba$
($D^b(\MA)$, $D^{-}(\MA)$) has precisely these factors, up to
ordering and derived equivalence. \label{JHhereditary}
\end{cor}

\medskip

\subsection{Weighted projective lines and canonical algebras}

Recall that a finite dimensional algebra $A$ over a field $k$ is
called {\em piecewise hereditary}, if there exists a hereditary and
abelian category $\mathcal{H}$ such that the bounded derived
categories $\dba$ and $D^b(\mathcal{H})$ are equivalent as
triangulated categories. In other words, there exists a tilting complex $T$ in
$D^b(\mathcal{H})$ with endomorphism ring being $A$.

\medskip

In order to proceed inductively, we need the following result, which
will follow immediately from Lemma \ref{inv-sgldim} below.
Another proof can be based on \cite[Corollary 3]{RComp}, where it is
shown that a finite dimensional algebra $A$ over a field is
piecewise hereditary if and only if for each indecomposable object
$X$ in $\dba$, there is no `path'  from $X[1]$ to $X$.

\begin{prop}  \label{recollph}

Suppose there is a recollement of finite dimensional algebras
$$\xymatrix@!=5pc{\dbb \ar[r] & \dba \ar@<+1.5ex>[l]
\ar@<-1.5ex>[l] \ar[r] & \dbc \ar@<+1.5ex>[l] \ar@<-1.5ex>[l]}.$$
If $A$ is piecewise hereditary, then $B$ and $C$ are also
piecewise hereditary. \label{reco-piecewise}

\end{prop}

\bigskip

A direct consequence is  the following analogue of
\cite[Corollary III.6.5]{H1}, where it has been shown that the
endomorphism algebra of a partial tilting module over a finite
dimensional hereditary algebra is a tilted algebra.

\begin{cor} The endomorphism algebra of a partial tilting
complex over a finite dimensional hereditary algebra is piecewise
hereditary. \label{endo-piecewise}

\end{cor}

\begin{proof} Let $T$ be a partial tilting complex over $A$, a
finite dimensional hereditary algebra. By Proposition
\ref{finitiltreco} $(1)\Rightarrow (3)$, it induces a recollement
$$\xymatrix@!=5pc{\dbb \ar[r] & \dba \ar@<+1.5ex>[l]
\ar@<-1.5ex>[l] \ar[r] & \dbc \ar@<+1.5ex>[l] \ar@<-1.5ex>[l]}$$
where $B$ and $C$ are finite dimensional algebras and
$C=\End_A(T)$. The statement follows then from
Proposition \ref{reco-piecewise}.
\end{proof}

\medskip

Recall the definition of strong global dimension (\cite{S},
\cite{HZ}). Let $A$ be a finite dimensional algebra. We define the
{\em length} of a compact complex $X\in K^b(P_A)$, denoted by
$\length(X)$, to be the length of its minimal $K^b(P_A)$-representative.
More precisely, suppose $$0 \ra P_{-s} \ra P_{-s+1} \ra \ldots \ra P_{r-1} \ra P_r \ra 0$$
is the minimal $K^b(P_A)$-representative of $X$ (where $P_i$ are
finitely generated projective modules and $-s\le r$ are integers).
Then $\length(X):=s+r$. The {\em strong global dimension} of $A$,
denoted by $\sgldim(A)$, is defined to be the supremum  of the
lengths of all indecomposable compact complexes over $A$. If $A$ has
finite strong global dimension, then it has finite global dimension.
Happel and Zacharia \cite[Theorem 3.2]{HZ}  have shown that $A$ is piecewise
hereditary if and only if it has finite strong global dimension.

\medskip

The following is a partial analogue of \cite[Lemma 2.1]{W} (for global
dimension) and \cite[3.3]{H} (for finitistic dimension).

\begin{lem} Suppose there is a recollement of finite dimensional algebras
$$\xymatrix@!=5pc{\dbb \ar[r] & \dba \ar@<+1.5ex>[l]
\ar@<-1.5ex>[l] \ar[r] & \dbc \ar@<+1.5ex>[l] \ar@<-1.5ex>[l]},$$
where the algebra $A$ has finite strong global dimension. Then the algebras
$B$ and $C$ also have finite strong global dimensions.
\label{inv-sgldim}
\end{lem}

\begin{proof} Assume $A$ has finite strong global
dimension say $d$. We will show $\sgldim(B) < \infty$ (and the
proof for $C$ is similar). By Corollary \ref{perfect}, the full
embedding $i_*:\dbb \ra \dba$ restricted to the perfect
subcategories $i_*: K^b(P_B) \ra K^b(P_A)$. Take an arbitrary
indecomposable complex $X$ in $K^b(P_B)$ with a minimal projective
resolution
$$0 \ra P_{-s} \ra P_{-s+1} \ra \ldots \ra P_{r-1} \ra P_r \ra 0$$
where $P_i$ are finitely  generated projective $B$-modules and
$-s\le r\in\Z$. We claim that $$r=\max\{n:\ \Hom_{\dbb}(B,X[n])\neq
0\},$$ $$s=\max\{n:\ \Hom_{\dbb}(X,B[n])\neq 0.\}$$ Indeed, the
first equality is implied by $\Hom_{\dbb}(B,X[n]) \cong H^n(X)$.
Moreover it is clear that $\Hom_{\dbb}(X,B[n]) \cong
\Hom_{K^b(P_B)}(X,B[n])$, which is trivial whenever $n>s$. To see
that $\Hom_{K^b(P_B)}(X,B[s])$ does not vanish, one takes a map
$f:P_{-s} \ra B$ which is identity restricted to a common
indecomposable direct summand of $P_{-s}$ and $B$, and is zero
elsewhere.

\medskip

Since $i_*$ is a full embedding, $i_*(X)$ is again
indecomposable and hence $\length(i_*(X)) \le \sgldim(A) =
d$. Since $i_*(B)$ is compact, it has finite length say $t$. Up to
shift (which does not change the length of a complex), we assume
the nonzero components of $i_*(X)$ are concentrated in positions
between $0$ and $d$, and those of $i_*(B)$ are between $k$ and
$k+t$ for some integer $k$. Therefore
$$\max\{n:\ \Hom_A(i_*(X),i_*(B)[n])\neq 0 \} \le k+t,$$
$$\max\{n:\ \Hom_A(i_*(B),i_*(X)[n])\neq 0 \} \le d-k.$$
But $\Hom_B(X,B[n]) \cong \Hom_A(i_*(X),i_*(B)[n])$ and
$\Hom_B(B,X[n]) \cong \Hom_A(i_*(B).i_*(X)[n])$. Hence $s \le k+t$
and $r\le d-k$. By definition $$\length(X)=s+r \le d+t
=\sgldim(A)+\length(i_*(B)).$$ Then $X$ being arbitrary implies that
$\sgldim(B)\le \sgldim(A)+\length(i_*(B))$, in particular it is
finite.
\end{proof}

\medskip

In contrast to the situation for global and finitistic dimension, the
converse of the statement is unfortunately wrong. For an example
we choose the quasi-hereditary algebra $A$ in \ref{positivexample}
given by
\[
\begin{picture}(100,10)
\put(-20,2){$\cdot$} \put(-26,2){\footnotesize $1$}
\put(22,2){$\cdot$} \put(28,2){\footnotesize $2$}
\put(-14,7){\vector(1,0){33}} \put(0,8){\footnotesize $\alpha$}
\put(18,1){\vector(-1,0){33}} \put(0,-7){\footnotesize $\beta$}
\put(60,2){$[\beta\circ\alpha=0].$}
\end{picture}
\]
It has infinite strong global dimension, for there exist compact
complexes of arbitrary length
$$\ldots \ra P(2) \ra P(2) \ra \ldots \ra P(2)\ra P(1).$$
But the quasi-hereditary structure gives a standard recollement,
where $D^b(\text{mod-}k)$ is on both sides.

\bigskip

\medskip

Now we are ready to prove the general  Jordan H\"older theorem
for bounded derived categories of finitely generated modules over
piecewise hereditary algebras over arbitrary base fields.

\begin{thm} Let $A$ be a finite dimensional piecewise hereditary algebra
over a field $k$. let $S_1, \ldots, S_n$ be the representatives of
isomorphism classes of simple $A$-modules. Denote by $D_i$ the
endomorphism rings of $S_i$ ($1\le i \le n$). Then $\dba$ has a
stratification with $D^b(\mathrm{mod}\text{-}D_i)$ ($1\le i \le n$) being the
factors. Moreover, any stratification of $\dba$ has precisely
these factors, up to derived equivalence. \label{JHpiecewise}

\end{thm}

\begin{proof}

Without loss of generality we may assume the algebra $A$ and the hereditary
category $\Hcal$ to be connected. Moreover, replacing $A$ by a derived
equivalent algebra, if necessary, we may assume - by \cite{HR,L} - that
$\Hcal$ either is $\mh$, the module category of a finite
dimensional hereditary $k$-algebra $H$, or it is $\coh(X)$, the category
of coherent sheaves on an exceptional curve $X$ (which is a weighted
projective line in the sense of \cite{GL}
when $k$ is algebraically closed). In the second case,
there is a `standard' tilting object
$T$ in $\Hcal=\coh(X)$ with endomorphism ring being a canonical
algebra in the sense of \cite{R1} (see for example \cite[2.4]{KM}).
To summarise: the algebra $A$ is derived equivalent to an indecomposable hereditary
algebra or to an indecomposable canonical algebra.






















\medskip


Recall that an object in $D^b(\Hcal)$ is called exceptional if  has no self-extension. A sequence of indecomposable and  exceptional objects $(E_1,E_2, \ldots, E_m)$ is called exceptional, if
$\Hom(E_i,E_j)=0=\Ext^1(E_i,E_j)$ for all $i>j$. An exceptional
sequence is called complete if the length $m$ equals to the rank $n$
of $A$ (i.e. the number of non-isomorphic simple modules). As
$\Hcal$ is hereditary, using the method of \cite[4.1,4.2]{HRi},
the indecomposable direct summands of a partial tilting
complex can be rearranged into a exceptional sequence (c.f.
\cite[2.5]{AKL2}). Moreover this exceptional sequence is complete if
and only if the partial tilting complex is a full tilting complex.

\medskip

On the set of complete exceptional sequences in $D^b(\Hcal)$ there
is an action of $\Z^n\ltimes B_n$ , where $B_n$ is the braid group
with $n-1$ generators acting by mutations. This action is moreover
transitive. In the case of hereditary algebras this has been shown by
\cite{R2} (extending the result for the algebraically closed case in \cite{CB}), and in the case
of exceptional curves by \cite{KM} (extending the result for the algebraically closed case in
\cite{M}). It follows that the list of endomorphism rings of the
indecomposable objects of a complete exceptional sequence in
$D^b(\Hcal)$ is an invariant. Therefore it is just the list
$(D_1,\ldots,D_n)$ of the endomorphism rings of non-isomorphism
simple $A$-modules.

\medskip

The existence of a stratification of $\dba \cong
D^b(\Hcal)$ as claimed follows from the
directedness of finite dimensional hereditary algebras and of canonical
algebras (or indeed of all piecewise hereditary algebras, since Happel's
argument in \cite[Lemma IV.1.10]{H1} works in general). Here, $A$
directed means that the quiver of $A$ has no oriented cycles, or equivalently
that $A$ has a simple projective module $eA$, for some idempotent
$e=e^2 \in A$ and the quotient algebra $A/AeA$ is again directed. The
two-sided ideal $AeA$ is semisimple and projective as a right module. Therefore,
the quotient map $A \rightarrow A/AeA$ is a homological epimorphism inducing
a recollement, which is a special case of the canonical recollement for a
quasi-hereditary algebra discussed in Example \ref{quasihereditary}. By
induction we get the stratification as claimed.
















\medskip

Uniqueness of the stratification will be shown by  induction on the number $n$ (the
rank of $A$) of isomorphism classes of simple $A$-modules. When
$n=1$, there is nothing to show. Now assume $n\geq 2$.
Given a recollement of $A$
$$\xymatrix@!=5pc{\dbb \ar[r] & \dba \ar@<+1.5ex>[l]
\ar@<-1.5ex>[l] \ar[r] & \dbc \ar@<+1.5ex>[l] \ar@<-1.5ex>[l]}$$ by
finite dimensional algebras $B$ and $C$, it follows from Corollary
\ref{reco-piecewise} and Subsection  \ref{Grothendieck}  that $B$
and $C$ are also piecewise hereditary, and hence directed, with rank
strictly smaller than $n$. It follows from the structure of the
recollement that the indecomposable direct summands of $j_!(C)
\oplus i_*(B)$ form a complete exceptional sequence in $\dba$. Note
that for directed algebras $B$ and $C$, the endomorphism ring of a
simple module is the same as the endomorphism ring of its projective
cover. Therefore the list of endomorphism rings of non-isomorphic
simple $C$-modules and non-isomorphic simple $B$-modules coincides
with that of non-isomorphic simple $A$-modules, i.e.
$\{D_1,\ldots,D_n\}$. The assertion follows by induction.
\end{proof}

\bigskip

The proof underlines the close link between recollements of bounded
derived categories and exceptional sequences. Exceptional sequences
have been used heavily by Bondal, Orlov and others when studying
derived categories. Recently, Ingalls and Thomas have
classified exceptional sequences by combinatorial objects (non-crossing partitions)
in certain situations related to tame quivers. This classification carries over
to all hereditary algebras, see \cite[Section 6]{Krause}, where exceptional sequences
are related also to thick subcategories that can occur in recollements of
derived categories of hereditary algebras.  In the case of hereditary algebras,
this provides an alternative point of view on stratifications.

\medskip

Note that the above proof is independent of \cite[Theorem
6.1]{AKL2}. But the proof there yields the
stronger fact that any stratification of $D(\MA)$ can be rearranged
into a chain of increasing derived module categories, via a sequence
of homological epimorphisms. We will obtain this stronger version
also for piecewise hereditary algebras. For that we have
to prove in the setting of \cite[Proposition 3.1]{AKL2}, that $G$ can be
chosen to be a piecewise hereditary algebra provided that $A$ is
piecewise hereditary. The case when $A$ is derived equivalent to a
hereditary algebra follows from \cite[Corollary 3.3]{AKL2}. Now
we consider the case when $A$ is derived equivalent to a canonical
algebra.

\medskip

\begin{prop} Let $A$ be a canonical algebra, and $E$ an indecomposable
exceptional object in $\dba$. Then there exists a piecewise
hereditary algebra $B$ that fits into a recollement of the form
$$\xymatrix@!=5pc{\dbb \ar[r] & \dba \ar@<+1.5ex>[l]
\ar@<-1.5ex>[l] \ar[r] & \dbc \ar@<+1.5ex>[l] \ar@<-1.5ex>[l]}$$
where $C=\End_A(E)$. The recollement is induced by a homological epimorphism.

\end{prop}

\begin{proof} Let $\dba = D^b(\Hcal)$ where $\Hcal=\coh(X)$
for some exceptional curve $X$. Without loss of generality we assume
$E$ lies in $\Hcal$.
By \cite[4.1]{HRi} any endomorphism of $E$ is either a
monomorphism or a epimorphism. If $E$ is a torsion sheaf, i.e. it has
finite length, then any endomorphism of $E$ must be an isomorphism.
If $E$ is a bundle, considering the rank (degree, respectively)
shows that any monomorphic (respectively, epimorphic) endomorphism of $E$ must be an
isomorphism. Therefore, the endomorphism ring of $E$ is a
skew-field.

\medskip

Let $T$ be the standard tilting object in $\Hcal$ with $A$ being the
endomorphism ring. Adjusting by using tubular mutation (see for
example \cite{L,M,ST,LP}), we can assume that
$\Hom_\Hcal(E,T)=0=\Ext^1\Hcal(T,E)$. Applying \cite[Proposition 6.5]{GL},
we obtain that the perpendicular category $\widehat{E}:=\{Y:
\Hom_\Hcal(E,Y)= 0 =\Ext^1_\Hcal(E,Y)\}$ of $E$ in $\Hcal$ admits a tilting object say $T'$. Indeed,
$T'=\ell(T)$ is constructed by the universal extension of $T$ and $E$: since
$\Hom_\Hcal(E,T)=0$ and $T$ is a tilting object, $\Ext^1_\Hcal(E,T)$ must be
nonzero, say of dimension $m$ over the skew-field $\End_\Hcal(E)$. Then the universal extension $$0 \ra T
\ra T' \ra E^{\oplus m} \ra 0$$ provides $T'=\ell(T)$ as the approximation of
$T$ in the perpendicular category $\widehat{E}$ (cf. the proof of \ref{adjoint} or \cite[Appendix A.1]{AKL}). It is
straightforward to check that $\widehat{E}$ is a hereditary and abelian
subcategory of $\Hcal$.

\medskip

Write $E^\bot:=\{Y\in D^b(\Hcal):\Hom_{D^b(\Hcal)}(E,Y[k])=0, \, \forall\ k \in
\Z\}$ for the perpendicular category of $E$ in the bounded derived
category. Since $\Hcal$ is hereditary, it is clear that
$E^\bot=\{Y[k]: Y\in \widehat{E},\,k\in\Z\} \cong
D^b(\widehat{E})$, which is equivalent to $\dbb$ for $B =
\End(T')$ as triangulated categories, since $T'$ is a tilting
object. As a compact exceptional object in the unbounded derived
module category $\dua$, $E$ generates a recollement on the
unbounded derived category level
$$\xymatrix@!=5pc{\dub \ar[r] & \dua \ar@<+1.5ex>[l]
\ar@<-1.5ex>[l] \ar[r] & \duc \ar@<+1.5ex>[l] \ar@<-1.5ex>[l]}$$
where $C=\End(E)$ (see the proof of Theorem
\ref{tiltreco}). Now the
global dimension of the canonical algebra $A$ is finite, and the
image of $B$ in $\dua$ is $T'$ which is compact. By Corollary
\ref{finitrestriction} such a recollement can be restricted to
$D^b(\rfmod)$ level.

\medskip

For the last statement, it suffices to prove that $i^*(A)$ is an exceptional object, see \cite[1.7]{AKL}.
By construction, the universal extension $0 \ra T
\ra T' \ra E^{\oplus m} \ra 0$ gives rise to the canonical triangle $$j_!j^!(A)\to A\to i_*i^*(A)\to j_!j^!(A)[1]$$ in which $i_*i^*(A)\cong T'$.
Since $i_*$ is fully faithful and $T'$ is a tilting object in $\Hcal$, we obtain $\Hom_B(i^*(A),i^*(A)[n])\cong \Hom_{D^b(\Hcal)}(T',T'[n])=0$ for all $n\not=0$.
\end{proof}

\medskip

\begin{rem}

We have shown the perpendicular category of an indecomposable
exceptional sheaf $E$ in $\Hcal=\coh(X)$ is derived equivalent to a
quasitilted algebra. When $E$ is a bundle, H\"ubner \cite[Theorem
5.4]{Hue} shows the perpendicular category $\widehat{E}$ is
equivalent to the module category of some hereditary algebra. When
$E$ is a simple torsion sheaf, Geigle and Lenzing \cite{GL} (see
also \cite[Example 3.2]{AKL}) showed $\widehat{X}$ is equivalent to
the category of coherent sheaves on another exceptional curve with
reduced weights.



\end{rem}

\bigskip

More generally, take an object $E$ in $D^b(\coh(X))$ without
self-extensions (not necessarily indecomposable). Without loss of
generality we assume it is multiplicity-free. Then its
indecomposable direct summands can be ordered into an exceptional
sequence. 
It follows by induction that there exists some piecewise
hereditary algebra $B$ fitting into a recollement
$$\xymatrix@!=5pc{\dbb \ar[r] & \dba \ar@<+1.5ex>[l]
\ar@<-1.5ex>[l] \ar[r] & \dbc \ar@<+1.5ex>[l] \ar@<-1.5ex>[l]}$$
where $C=\End_A(E)$ (c.f. proof of \cite[Corollary 3.3]{AKL2}). Now in
the setting of \cite[Proposition 3.1]{AKL2}, if we start with a
piecewise hereditary algebra $A$, then $G$ can be chosen to be again
piecewise hereditary (in particular, it is an ordinary algebra). Therefore
any stratification of $A$ can be rearranged into a chain of
increasing derived module categories corresponding to homological epimorphisms (c.f. proof of \cite[Theorem
6.1]{AKL2}).

\medskip

At this point we have proved (1) $\Rightarrow$ (3) in Proposition
\ref{hered} for piecewise hereditary algebras. As piecewise
hereditary algebras have finite global dimension, the equivalences (1) $\Leftrightarrow$ (3) $\Leftrightarrow$
(4) $\Leftrightarrow$ (5) $\Leftrightarrow$ (6) in Proposition
\ref{hered} hold true. Combining this information with Theorem \ref{JHpiecewise}, we
obtain the Jordan H\"older theorem for piecewise hereditary algebras
on different levels.

\begin{cor} Let $A$ be a piecewise hereditary algebra over a field
$k$. Then Theorem \ref{JHpiecewise} holds true also for recollements on
$D^{-}(\text{Mod}-)$, $D^b(\text{Mod}-)$ and $D(\text{Mod}-)$
levels.
\end{cor}

\medskip

Now Corollaries \ref{number} and \ref{endo-piecewise} extend to piecewise hereditary algebras. We obtain:

\begin{cor} The endomorphism algebra of a partial tilting complex of a piecewise
hereditary algebra is again piecewise hereditary.

\end{cor}

\begin{proof} Let $A$ be a piecewise hereditary algebra and $X$ a
partial tilting complex. So $X$ is exceptional in $\dba$. We have
just shown that $X$ determines a recollement of $A$ with
$C:=\End(X)$ on the right hand side. It follows from Corollary
\ref{reco-piecewise} that $C$ is piecewise hereditary.
\end{proof}

\bigskip

{\bf Acknowledgements:} The first named author acknowledges partial support from Universit\`a di Padova through Project CPDA105885/10  "Differential graded categories", by the DGI and the European Regional Development Fund, jointly, through Project MTM2008--06201--C02--01, and by the Comissionat per Universitats i Recerca of the Generalitat de Catalunya, Project 2009 SGR 1389.

\end{document}